\documentclass{cocv}
\usepackage{amssymb}
\usepackage{amsmath}
\usepackage[T1]{fontenc}
\usepackage[latin1]{inputenc}
\usepackage{amsfonts}
\usepackage{graphicx}
\usepackage{epsfig}
\usepackage{url}

\newtheorem{theorem}{Theorem}
\newtheorem{lemma*}{Lemma}

\newtheorem{proposition}{Proposition}

\begin{document}
\title{Boundary Control of Reaction-Diffusion PDEs\\ on Balls in Spaces of Arbitrary Dimensions}
%
\author{Rafael Vazquez}\address{Department of Aerospace Engineering, Universidad de Sevilla, Camino de los Descubrimiento s.n., 41092 Sevilla, Spain.}
\author{Miroslav Krstic}\address{Department of Mechanical and Aerospace Engineering, University of California San Diego, La Jolla, CA 92093-0411, USA.}
\date{October 2015}
\begin{abstract} An explicit output-feedback boundary feedback law is introduced that stabilizes an unstable linear constant-coefficient reaction-diffusion equation on an $n$-ball (which in 2-D reduces to a disk and in 3-D reduces to a sphere) using only measurements from the boundary. The backstepping method is used to design both the control law and a boundary observer. To apply backstepping the system is reduced to an infinite sequence of 1-D systems using spherical harmonics. Well-posedness and stability are proved in the $H^1$ space. The resulting control and output injection gain kernels are the product of the backstepping kernel used in control of one-dimensional reaction-diffusion equations and a function closely related to the Poisson kernel in the $n$-ball.
 \end{abstract}
%
%
\subjclass{93B51 / 93B52 / 94C05 / 93C20 / 93D05 / 93D15}
\keywords{parabolic equations / control design / observer design / spherical harmonics / infinite-dimensional backstepping}
\maketitle
\section*{Introduction}


In this paper we present an explicit full-state boundary feedback law to stabilize an unstable linear constant-coefficient reaction-diffusion equation on an $n$-ball (which in 2-D reduces to a disk and in 3-D reduces to a sphere).

We introduce the following tools in the paper. Based on the domain shape, we employ ultraspherical coordinates (which in 2-D are polar coordinates and in 3-D spherical coordinates), and then transform the system into an infinite sequence of independently-controlled 1-D systems by using spherical harmonics. For each harmonic we design a feedback law using the backstepping method~\cite{krstic}, which allows us to obtain exponential stability of the origin in the $H^1$ norm. The same procedure is used to produce an observer using boundary measurements (a weighted spatial average of the state could also be considered, see for instance~\cite{tsubakino}). The combination of both designs produces the final result.

This result extends~\cite{disk} and~\cite{disk2}, which describe, respectively, the (full-state) control design for the particular case of a 2-D disk and a general $n$-ball; that same design, augmented with an observer, is applied in~\cite{jie} to multi-agent deployment in 3-D space, with the agents distributed on a disk-shaped grid and commanded by leader agents located at the boundary. Older, related results that use backstepping include the design an output feedback law for a convection problem on an annular domain (see~\cite{convloop}, also~\cite{chengkang-xie}), or observer design on cuboid domains~\cite{kugi}. However, going from an annular domain to a disk (which includes the origin) complicates the design, as (apparent) singularities appear on the equations and have to be dealt with. This difficulty also shows up in the 3-D and higher-dimensional settings. Despite this complication, we are able to explicitly find the backstepping kernels and subsequently reconstruct our control law back in physical space, showing  that the closed-loop system is well-posed and exponentially stable in the $H^1$ norm. The resulting output-feedback control law is remarkable, in the sense that both the controller and the output injection gains are formulated as a multiple integral whose kernel is a product of two factors. The first factor only has radial dependence and it is exactly the same kernel that appears when applying backstepping to one-dimensional reaction-diffusion equations. The second factor depends on both radius and angle and is closely related to the Poisson kernel in the $n$-ball, which is a function used to solve Laplace's problem in that same domain. This structure greatly simplifies the proof of $H^1$ stability. The interest of having a result in $n$ dimensions, beyond having a general formula useful for the more conventional 2-D and 3-D settings, lies in emerging applications which can be modelled by high-dimensional diffusion. For instance, the use of reaction-diffusion equations to model multi-agent systems~\cite{meurer} has found much success, as it allows exploiting PDE control methods in applications such as agent motion planning; in this context, higher-dimensional models might contribute to related problems such as control of social networks and opinion dynamics.

The backstepping method has become ubiquitous in PDE control, as a powerful and elegant tool with many other applications including, among others, flow control~\cite{vazquez,vazquez-coron}, crystal growth processes~\cite{dubljevic}, nonlinear PDEs~\cite{vazquez2}, coupled parabolic systems~\cite{orlov}, hyperbolic 1-D systems~\cite{vazquez-nonlinear,florent,krstic3,Vazquez2014}, adaptive control~\cite{krstic4}, wave equations~\cite{krstic2}, and delays~\cite{krstic5}. Other design methods are also applicable to the geometry considered in this paper (see for instace~\cite{triggiani},~\cite{meurer2} or~\cite{Barbu}). Also, there have been specific results on disk- or spherical-shaped domains, such as~\cite{Prieur} and~\cite{scott}, however these works assume perfect symmetry of initial conditions which allows to consider only radial variations, thus simplifying the problem considerably.

The structure of the paper is as follows. In Section~\ref{sec-plant} we introduce the problem and state our main result. We explain the basis of our design method (spherical harmonics) in Section~\ref{sec-ultra} and explicitly find the control kernels in Section~\ref{sec-design}. The observer is dealt with in Section~\ref{sect-observer}. Next, we prove the closed-loop stability in Section~\ref{sect-stability}. We conclude the paper with some remarks in Section~\ref{sec-conclusions}.

\section{Main result}\label{sec-plant}
Consider the following constant-coefficient reaction-diffusion system in an $n$-dimensional ball of radius $R$:
\begin{equation}\label{eqn-ndim-rectangular}
\frac{\partial u}{\partial t}=\epsilon \left(\frac{\partial^2 u}{\partial x_1^2}+\frac{\partial^2 u}{\partial x_2^2}+\hdots+\frac{\partial^2 u}{\partial x_n^2}\right)+\lambda u=\epsilon \bigtriangleup_n u+\lambda u,
\end{equation}
where $\epsilon,\lambda$ are positive numbers, $u=u(t,\vec x)$, with $\vec x=[x_1,x_2,\hdots,x_n]^T$, is the state variable, evolving for $t>0$ in the $n$-ball $B^n(R)$ defined as
\begin{equation}
B^n(R)=\left\{\vec x\in \mathbb R^n:\Vert \vec x \Vert \leq R\right\},
\end{equation}
with boundary conditions on the boundary of $B^n(R)$, which is the $(n-1)$-sphere $S^{n-1}(R)$ defined as
\begin{equation}
S^{n-1}(R)=\left\{\vec x \in \mathbb R^n:\Vert \vec x \Vert = R\right\}.
\end{equation}
The boundary condition is assumed to be of Dirichlet type,
\begin{equation}\label{eqn-ndim-bc}
u(t,\vec x)\Big|_{\vec x\in S^{n-1}(R)}=U(t,\vec x),
\end{equation}
where $U(t,\vec x)$ is the actuation. On the other hand the measurement $y(t,\vec x)$ is defined as
\begin{equation}\label{eqn-measurement}
y(t,\vec x)=\partial_r u(t,\vec x)\Big|_{\vec x\in S^{n-1}(R)},
\end{equation}
where $\partial_r$ denotes the derivative in the radial direction (normal to the $(n-1)$-sphere), which is defined as $\partial_r u(t,\vec x)=\vec \nabla u \cdot \frac{\vec x}{\Vert \vec x \Vert}$.

Define now the $L^2$ norm and $H^1$ norm of a scalar function in the $n$-Ball as
\begin{equation}\label{eqn-defL2}
\Vert f \Vert_{L^2}=\left( \int_{B^n(R)}f^2(\vec x)
d\vec x\right)^{1/2},\quad \Vert f \Vert_{H^1}=\left(\Vert f \Vert_{L^2}^2+ \int_{B^n(R)} \Vert \vec \nabla f (\vec x) \Vert ^2
d\vec x\right)^{1/2}
\end{equation}
We denote simply by $L^2$ (resp. $H^1$) the space of square-integrable functions (resp. of functions with square-integral gradient) over the $n$-ball of radius $R$. Finally denote by $H^1_0$ the space of $H^1$ functions vanishing at the boundary (in the usual sense of traces, see e.g.~\cite[p.259]{evans}).

The following theorem was proved in~\cite{disk2}.
\begin{theorem}\label{th-main}
Consider (\ref{eqn-ndim-rectangular}) and (\ref{eqn-ndim-bc}) with initial conditions $u_0(\vec x)$ and the following (explicit) full-state feedback law for $U$:
\begin{eqnarray}
 U&=&- \frac{1}{\mathrm{Area}(S^{n-1}) }\sqrt{\frac{\lambda}{\epsilon}} \int_{B^n(R)}  \mathrm{I}_1\left[\sqrt{\frac{\lambda}{\epsilon}(R^2-\Vert \vec\xi \Vert ^2)}\right]  
 \frac{\sqrt{R^2-\Vert \vec \xi \Vert ^2}}{\Vert \vec x - \vec \xi \Vert^n} u(t,\vec \xi) d \vec\xi,\quad \label{control}
\end{eqnarray}
where the integral in (\ref{control}) is extended over the whole $n$-ball of radius $R$ and $\vec x\in S^{n-1}(R)$. In (\ref{control}), $\mathrm{I}_1$ is the first-order modified Bessel function of the first kind, and \begin{equation}\label{eqn-surf}
\mathrm{Area}(S^{n-1})=\frac{2 \pi^{\frac{n}{2}}}{\Gamma\left(\frac{n}{2}\right)}
\end{equation}
 is the ``surface area'' of the unit $(n-1)$-sphere, with $\Gamma$ being the Euler Gamma function (see~\cite{abramowitz}, p.253).
Assume in addition that $u_0\in L^2$. Then system (\ref{eqn-ndim-rectangular}),(\ref{eqn-ndim-bc}) has a unique $L^2$ solution, and the equilibrium profile $u\equiv 0$ is exponentially stable in the $L^2$ norm, i.e., there exists $c_1,c_2>0$ such that
\begin{equation}
\Vert u(t,\cdot) \Vert_{L^2} \leq c_1 \mathrm{e}^{-c_2t} \Vert u_0 \Vert_{L^2}.
\end{equation}
\end{theorem}

In Theorem~\ref{th-main}, feedback law (\ref{control}) assumes that the full state is known. This paper extends this result to the situation where only boundary measurements are used, by employing an observer. Denote by $\hat u$ the state of the observer (which approximates $u$). Then, $\hat u$ is computed as the solution of the following PDE
\begin{eqnarray}
\hat u_t&=& \epsilon \bigtriangleup_n \hat u+\lambda \hat u
- \frac{\sqrt{ \lambda \epsilon} }{\mathrm{Area}(S^{n-1}) }   \mathrm{I}_1\left[\sqrt{\frac{\lambda}{\epsilon}(R^2-\Vert \vec x \Vert ^2)} \right]  \sqrt{R^2-\Vert \vec x \Vert ^2}  \int_{S^{n-1}(R)}  \frac{y(t,\vec \xi)-\hat u_r(t,\vec \xi)}{\Vert \vec x - \vec \xi \Vert^n}
d\vec \xi,\label{eqn-ndim-obs-rect-th}
\end{eqnarray}
with boundary conditions
\begin{equation}\label{eqn-ndim-bc-obs}
\hat u(t,\vec x)\Big|_{\vec x\in S^{n-1}(R)}=U(t,\vec x).
\end{equation}
Notice that (\ref{eqn-ndim-obs-rect-th})--(\ref{eqn-ndim-bc-obs}) is a copy of (\ref{eqn-ndim-rectangular})--(\ref{eqn-ndim-bc}) with additional output injection in the right-hand side of (\ref{eqn-ndim-obs-rect-th}). The following result extends Theorem~\ref{th-main} by using the observer states in the feedback law (\ref{control}).
\begin{theorem}\label{th-main2}
Consider (\ref{eqn-ndim-rectangular})--(\ref{eqn-ndim-bc}) and (\ref{eqn-ndim-obs-rect-th})--(\ref{eqn-ndim-bc-obs}) with initial conditions $u_0(\vec x)$ and $\hat u_0(\vec x$) respectively, and the the following (explicit) full-state feedback law for $U$:
\begin{eqnarray}
 U&=&- \frac{1}{\mathrm{Area}(S^{n-1}) }\sqrt{\frac{\lambda}{\epsilon}} \int_{B^n(R)}  \mathrm{I}_1\left[\sqrt{\frac{\lambda}{\epsilon}(R^2-\Vert \vec\xi \Vert ^2)}\right]  
 \frac{\sqrt{R^2-\Vert \vec \xi \Vert ^2}}{\Vert \vec x - \vec \xi \Vert^n} \hat u(t,\vec \xi) d \vec\xi.\quad \label{control-obs}
\end{eqnarray}
Assume in addition that $u_0\in H^1_0$ and $\hat u_0\equiv 0$. Then the augmented system ($u,\hat u$) has an unique $H^1$ solution, and the equilibrium profile $u,\hat u \equiv 0$ is exponentially stable in the $H^1$ norm, i.e., there exists $c_1,c_2>0$ such that
\begin{equation}
\Vert u(t,\cdot) \Vert_{H^1}+\Vert \hat u(t,\cdot) \Vert_{H^1} \leq c_1 \mathrm{e}^{-c_2t} \Vert u_0 \Vert_{H^1}.
\end{equation}
\end{theorem}

Theorem~\ref{th-main2} assumes $H^1_0$ initial conditions for $u$ (which would be the natural open-loop boundary condition verifying the $0$-th order compatibility conditions, see e.g.~\cite[p.365]{evans}), and identically zero initial conditions for $\hat u$. It will be seen that this guarantees that the $0$-th order compatibility conditions are verified for the augmented system. Other combinations of initial conditions are possible if the $0$-th order compatibility conditions are verified. If necessary, control law (\ref{control-obs}) could be modified by adding extra decaying terms to enforce the satisfaction of the compatibility conditions (see e.g.~\cite{vazquez-nonlinear}).

In the next sections we sketch the proof of the result. The basic tool used to design the controller and observer is the theory of spherical harmonics, which is briefly reviewed in Section~\ref{sec-ultra}. Then, in Section~\ref{sec-design} we explain the rationale behind the method used for reaching control law (\ref{control}), and  in Section~\ref{sect-observer} we construct the observer (\ref{eqn-ndim-obs-rect-th})--(\ref{eqn-ndim-bc-obs}). Finally Section~\ref{sect-stability} contains the proof of stability and well-posedness for the closed loop augmented system.

\section{Ultraspherical coordinates and Spherical harmonics}\label{sec-ultra}

Equation (\ref{eqn-ndim-rectangular}) can be written in $n$-dimensional spherical coordinates, also known as ultraspherical coordinates (see~\cite{atkinson}, p. 93), which consist of one radial coordinate and $n-1$ angular coordinates, namely $(r,\theta_1,\theta_2,\hdots,\theta_{n-1})$, where $r\in[0,R]$ is the radial coordinate, and the angular coordinates are $\theta_1 \in[0,2\pi]$ and $\theta_i\in[0,\pi]$ for $2\leq i\leq n-1$. Using these coordinates, the rectangular coordinates are
\begin{eqnarray}
x_1&=&r\cos \theta_1 \sin \theta_2 \sin \theta_3 \hdots \sin\theta_{n-1},\\
x_2&=&r\sin \theta_1 \sin \theta_2 \sin \theta_3 \hdots \sin\theta_{n-1},\\
&\vdots & \nonumber \\
x_{n-1}&=&r\cos \theta_{n-2} \sin\theta_{n-1},\\
x_n&=&r\cos \theta_{n-1}.
\end{eqnarray}
For instance the usual spherical coordinates for the 3-ball (sphere) are
$x_1=r\cos \theta_1\sin \theta_2$, $x_2=r\sin \theta_1\sin \theta_2$, and $x_3=r \cos \theta_2$.

Following~\cite{atkinson} (p. 94), the  definition of the Laplacian $\bigtriangleup_n$ in ultraspherical coordinates is as follows:
\begin{equation}
\bigtriangleup_n=\frac{1}{r^{n-1}}\frac{\partial}{\partial r} \left(r^{n-1} \frac{\partial}{\partial r} \right) +\frac{1}{r^2} \bigtriangleup^*_{n-1}
\end{equation}
where $\bigtriangleup^*_{n-1}$ is called the Laplace-Beltrami operator and represents the Laplacian over the $(n-1)$-sphere. Its definition is recursive, as follows:
\begin{eqnarray}
\bigtriangleup^*_{1}&=&\frac{\partial^2}{\partial \theta_1^2}, \quad
\bigtriangleup^*_{n}= \frac{1}{\sin^{n-1} \theta_n} \frac{\partial}{\partial \theta_{n}}
  \left(\sin^{n-1} \theta_n \frac{\partial}{\partial  \theta_{n}} \right) +\frac{\bigtriangleup^*_{n-1}}{\sin^{2} \theta_n}, \qquad
\end{eqnarray}
As example, we show $\bigtriangleup^*_{2}$:
\begin{eqnarray}\label{eqn-laplacians}
\bigtriangleup^*_{2}&=&  \frac{1}{\sin \theta_2} \frac{\partial}{\partial \theta_{2}}
 \left(\sin \theta_2 \frac{\partial}{\partial  \theta_{2}} \right) +\frac{1}{\sin^{2} \theta_2} \frac{\partial^2}{\partial \theta_1^2}.
\end{eqnarray}
Denote, for simplicity, $\vec \theta=[\theta_1,\hdots,\theta_{n-1}]^T$.
Thus, Equation (\ref{eqn-ndim-rectangular})  can be written in ultraspherical coordinates as follows
\begin{eqnarray}\label{eqn-ndim-spherical}
\partial_t u&=& \frac{\epsilon}{r^{n-1}}\partial_r \left(r^{n-1} \partial_r u \right) +\frac{1}{r^2} \bigtriangleup^*_{n-1} u+\lambda u,
\end{eqnarray}
with boundary conditions
\begin{equation}\label{eqn-ndim-spherical-bc}
u(t,R,\vec \theta\,)=U(t,\vec \theta\,),
\end{equation}
and measurement
\begin{equation}\label{eqn-ndim-spherical-obs}
y(t,\vec \theta\,)=u_r(t,R,\vec \theta\,).
\end{equation}
In what follows $U$ and $y$ will be written without arguments for the sake of simplicity.

\subsection{Expansion in Spherical Harmonics}\label{sect-expansion}

To handle the angular dependency in (\ref{eqn-ndim-spherical}), we expand both the  $u$ and $U$ using a (complex-valued) Fourier-Laplace series of Spherical Harmonics\footnote{Spherical harmonics were introduced by Laplace to solve the homonymous equation and have been widely used since, particularly in geodesics, orbital mechanics, electromagnetism and computer graphics. A very complete treatment on the subject can be found in~\cite{atkinson}.} for the $n$-ball:
\begin{eqnarray}\label{eqn-expand}
u(t,r,\vec \theta)&=&\sum_{l=0}^{l=\infty} \sum_{m=0}^{m=N(l,n)} u^{m}_{l}(r,t) 
Y_{lm}^n(\vec \theta),\quad\\
U(t,\vec \theta)&=&\sum_{l=0}^{l=\infty}  \sum_{m=0}^{m=N(l,n)}U^{m}_{l}(t)  
Y_{lm}^n(\vec \theta),\quad\label{eqn-expand2}
\end{eqnarray}
where $N(l,n)$ is the number of (linearly independent) $n$-dimensional spherical harmonics of degree $l$, given by $N(0,n)=1$ (representing the mean value over the $n$-ball) and, for $l>0$, 
\begin{equation}
N(l,n)=\frac{2l+n-2}{l}\left( \begin{array}{c} l+n-3 \\ l-1 \end{array} \right),
\end{equation}
with $Y_{lm}^n$ being the $m$-th $n$-dimensional spherical harmonic of degree $l$. The coefficients in the series of (\ref{eqn-expand})--(\ref{eqn-expand2}) are possibly complex-valued and defined as
\begin{eqnarray}\label{eqn-harm-coef}
u^{m}_{l}(r,t)&=& \int_{0}^{\pi} \hdots   \int_0^{\pi} \int_0^{2\pi} u (t,r,\vec \theta) \bar Y_{lm}^n(\vec \theta)  
\sin^{n-2} \theta_{n-1}
 \sin^{n-3} \theta_{n-2} \hdots \sin \theta_2 d\vec \theta,\qquad \\
U^{m}_{l}(t)&=& \int_{0}^{\pi} \hdots   \int_0^{\pi}\int_0^{2\pi}   U (t,\vec \theta) \bar Y_{lm}^n(\vec \theta)  
\sin^{n-2} \theta_{n-1}
 \sin^{n-3} \theta_{n-2} \hdots \sin \theta_2 d\vec \theta,\qquad
\end{eqnarray}
where $\bar Y_{lm}^n$ represent the complex conjugate of $Y_{lm}^n$, and the product of sines inside the integral represent the $(n-1)$-sphere ``area'' element. The spherical harmonics $Y_{lm}^n$ can be explicitly defined in terms of the associated Legendre functions, but their explicit expression is complicated and
%
not really necessary when dealing with a constant-coefficient system ($\lambda$ constant in (\ref{eqn-ndim-rectangular})). Only their properties will be used in what follows.

The following important property holds~\cite[p. 97]{atkinson}, which means that the $n$-dimensional spherical harmonics are eigenfunctions of the Laplacian over the $(n-1)$-sphere:
\begin{equation}\label{eqn-eigen}
\bigtriangleup^*_{n-1}  Y^n_{lm}=-l(l+n-2) Y_{lm}^n.
\end{equation}
Using (\ref{eqn-eigen}), developing all the terms of (\ref{eqn-ndim-spherical}) in spherical harmonics and using the fact that the harmonics are linearly independent of each other, we obtain that 
each coefficient $u^{m}_{l}(t,r)$, for $l\in\mathbb{N}$ and $0\leq m \leq N(l,n)$, verifies the following independent 1-D reaction-diffusion equation:
\begin{equation}\label{eqn-un}
\partial_t u^{m}_{l}=\frac{\epsilon}{r^{n-1}} \partial_r \left( r^{n-1} \partial_r u^{m}_{l} \right)-l(l+n-2) \frac{\epsilon}{r^2} u^{m}_{l}+\lambda u^{m}_{l},
\end{equation}
evolving in $r \in[0,R],\,t>0$, with  boundary conditions 
\begin{eqnarray}\label{eqn-bcun}
u^{m}_{l}(t,R)&=&U^{m}_{l}(t),
\end{eqnarray}
and measurement
\begin{eqnarray}\label{eqn-measurementun}
y^{m}_{l}(t)&=&\partial_r u^{m}_{l}(t,R),
\end{eqnarray}

Thanks to the fact that the coefficients are constant and thus independent of the angular coordinates, the equations are not coupled and thus we can independently design each $U^{m}_{l}$ and later assemble all of the them using (\ref{eqn-expand2}) to find an expression for $U$.

\section{Full-state control law design}\label{sec-design}
In this section we review how the full-state feedback law (\ref{control}) has been constructed, since it is applied in Theorem~\ref{th-main2} with the observer states. Starting from the representation in ultraspherical coordinates (\ref{eqn-ndim-spherical}), we stabilize each  harmonic independently by using the backstepping method. Finally, we put together all the harmonics using (\ref{eqn-expand2}), thus reconstructing the feedback law in physical space.

\subsection{Backstepping transformation}\label{sect-back}
Our approach to design $U^{m}_{l}(t)$ is to seek a mapping that transforms (\ref{eqn-un}) into the following target system
\begin{equation}\label{eqn-wn}
\partial_t w^{m}_{l}=\frac{\epsilon}{r^{n-1}} \partial_r \left( r^{n-1} \partial_r w^{m}_{l} \right)-l(l+n-2) \frac{\epsilon}{r^2} w^{m}_{l},
\end{equation}
a stable heat equation (a negative reaction coefficient could also be added if a more rapid convergence rate is desired),
with boundary conditions
\begin{eqnarray}\label{eqn-bcwn}
w^{m}_{l}(t,R)&=&0.
\end{eqnarray}

The transformation is defined as follows:
\begin{equation}\label{eqn-transf}
w^{m}_{l}(t,r)=u^{m}_{l}(t,r)-\int_0^r K^n_{lm}(r,\rho) u^{m}_{l}(t,\rho) d\rho,
\end{equation}
and then $U^{m}_{l}(t)$ is found by substituting the transformation (\ref{eqn-transf}) in (\ref{eqn-un}) an using (\ref{eqn-bcwn}). 

To find the kernel equations we proceed, as usual in the backstepping method~\cite{krstic}, by substituting both the original and target systems in the transformation and integrating by parts when possible. After a rather tedious calculation, we obtain the following PDE that the kernel must verify:
\begin{eqnarray}
&& \frac{1}{r^{n-1}} \partial_r \left(r^{n-1}   \partial_r K^n_{lmr} \right)
-  \partial_\rho \left(\rho^{n-1} \partial_\rho  \left( \frac{K^n_{lm}}{\rho^{n-1}} \right) \right)
-l(l+n-2) \left(\frac{1}{r^2}- \frac{ 1}{\rho^2} \right)K^n_{lm}=\frac{\lambda}{\epsilon} K^n_{lm}.\label{eqn-Kn}
\end{eqnarray} 
Also, the following boundary conditions have to be verified
\begin{eqnarray}
0&=& \lambda  + \epsilon  K^n_{lmr} (r,r)  +\frac{ \epsilon}{r^{n-1}}  \frac{\partial}{\partial r} \left(r^{n-1}   K^n_{lm}(r,r) \right) 
+\epsilon 
\left( \frac{K^n_{lm} (r,\rho)}{\rho^{n-1}}\right)_\rho \bigg|_{\rho=r} r^{n-1},\label{eqn-bcKn1}\\ 
0&=&K^n_{lm}(r,0),\label{eqn-bcKn2}\\
0&=&\lim_{\rho\rightarrow 0} \left(\partial_\rho \left( \frac{K^n_{lm} (r,\rho)}{\rho^{n-1}}\right) \rho^{n-1}\right).\label{eqn-bcKn3}
\end{eqnarray}
Developing boundary condition (\ref{eqn-bcKn1}), we obtain
\begin{eqnarray}
0&=& \lambda  +  2 \epsilon \frac{d}{dr}\left(K^n_{lm}(r,r)\right),
\end{eqnarray}
which integrates to
\begin{eqnarray}\label{eqn-bcKnint}
K^n_{lm}(r,r) 
&=& -\frac{\lambda  r}{2 \epsilon },
\end{eqnarray}
where we have used (\ref{eqn-bcKn1}) at $r=0$ (i.e., $K^n_{lm} (0,0)=0$).
Finally (\ref{eqn-bcKn1}) can be written as
\begin{eqnarray}\label{eqn-bcrara}
0&=&\partial_\rho  K^n_{lm} (r,0)- (n-1) \lim_{\rho\rightarrow 0}  \frac{K^n_{lm} (r,\rho)}{\rho} .
\end{eqnarray}
However, from the second boundary condition one obtains $K^n_{lm} (r,0)=0$. Thus, if we assume that $K^n_{lm}(r,\rho)$ is differentiable in $\rho$ at $\rho=0$, obviously  implies $\lim_{\rho\rightarrow 0}  \frac{K^n_{lm} (r,\rho)}{\rho}=\partial_\rho  K^n_{lm} (r,0)$. Thus boundary condition (\ref{eqn-bcrara}) is automatically verified for $n=2$ and implies, for $n>2$,
\begin{equation}\label{eqn-bcKrho}
(n-2)\partial_\rho  K^n_{lm} (r,0)=0.
\end{equation}

\subsection{Explicitly solving the kernel equations}
\label{sec-explicit}
To solve (\ref{eqn-Kn}) with boundary conditions (\ref{eqn-bcKnint}), (\ref{eqn-bcKn2}) and (\ref{eqn-bcKrho}), consider the change of variables $K^n_{lm} (r,\rho)=G^n_{lm} (r,\rho)\rho\left(\frac{\rho}{r}\right)^{l+n-2}$. We end up with
\begin{eqnarray}\label{eqn-gn}
\frac{\lambda}{\epsilon} G^n_{lm} &=& \partial_{rr} G^n_{lm} +(3-n-2l)\frac{ \partial_r G^n_{lm}}{r}
 -\partial_{\rho\rho} G^n_{lm}+(1-n-2l)\frac{ \partial_\rho G^n_{lm}}{\rho},\label{eqn-Gn}
\end{eqnarray}
with only one boundary condition, since (\ref{eqn-bcKn2}) and (\ref{eqn-bcKrho}) are automatically verified:
\begin{eqnarray}
G^n_{lm}(r,r) 
&=& - \frac{\lambda}{2 \epsilon }.\label{eqn-Gnbc}
\end{eqnarray}
now assume a solution $G^n_{lm}(r,\rho)$ of the form 
\begin{equation}
G^n_{lm}(r,\rho)=\Phi\left(\left(\frac{\lambda}{\epsilon}(r^2-\rho^2)\right)^{1/2}\right),
\end{equation}
noticing that $\Phi$ is independent of $n$. Expressing the derivatives of $G_n$ in terms of $\Phi$ and replacing them in (\ref{eqn-Gn}) we find
\begin{eqnarray}
\frac{\lambda}{\epsilon} \Phi''+3\frac{\lambda}{\epsilon}\left(\frac{\lambda}{\epsilon}(r^2-\rho^2)\right)^{-1/2}\Phi'=\frac{\lambda}{\epsilon} \Phi\label{eqn-Phii}
\end{eqnarray}
with boundary condition $\Phi(0)= - \frac{\lambda}{2 \epsilon }$.
Denoting  $x=\left(\frac{\lambda}{\epsilon}(r^2-\rho^2)\right)^{1/2}$ and the derivatives with respect to $x$ with a dot, we can write (\ref{eqn-Phii}) as
\begin{eqnarray}
 \ddot \Phi(x)+\frac{3}{x}\dot \Phi(x)-\Phi(x)=0,
\end{eqnarray}
and finally calling $\Psi(x)=x \Phi(x)$, we obtain
\begin{eqnarray}
\left(\frac{\ddot \Psi}{x}-2\frac{\dot \Psi}{x^2}+2\frac{\Psi}{x^3}\right)+\frac{3}{x}\left(\frac{\dot \Psi}{x}-\frac{\Psi}{x^2}\right)-\frac{\Psi}{x}=0,
\end{eqnarray}
which cross-multiplied by $x^3$ gives
\begin{eqnarray}
x^2\ddot \Psi +x\dot \Psi-(1+x^2)\Psi=0,
\end{eqnarray}
which is Bessel's modified differential equation of order 1 (see~\cite{abramowitz}, p. 374, Sec. 9.6.1). The (bounded) solution is
\begin{equation}
\Psi(x)=C_1 \mathrm{I}_1(x),
\end{equation}
where $I_1$ is the modified Bessel function of order 1, and going backwards we obtain
\begin{equation}
\Phi(x)=C_1 \frac{\mathrm{I}_1(x)}{x}.
\end{equation}
noticing that
$\lim_{x \rightarrow 0} \frac{\mathrm{I}_1(x)}{x} = 1/2,$
we get 
$C_1=-\frac{\lambda}{\epsilon}$ (from the boundary condition at $x=0$).
Therefore, we obtain, by undoing the change of variables to recover $G_n$,
\begin{equation}
G^n_{lm}(r,\rho)=-\frac{\lambda}{\epsilon} \frac{\mathrm{I}_1\left[\sqrt{\frac{\lambda}{\epsilon}(r^2-\rho^2)}\right]}{\sqrt{\frac{\lambda}{\epsilon}(r^2-\rho^2)}},
\end{equation}
and therefore
\begin{equation}
K^n_{lm}(r,\rho)=-\rho\left(\frac{\rho}{r}\right)^{l+n-2}  \frac{\lambda}{\epsilon} \frac{\mathrm{I}_1\left[\sqrt{\frac{\lambda}{\epsilon}(r^2-\rho^2)}\right]}{\sqrt{\frac{\lambda}{\epsilon}(r^2-\rho^2)}}.\label{eqn-Knformula}
\end{equation}
\subsection{Finding the n-D backstepping kernel}
The backstepping transformation (\ref{eqn-transf}), written in real space by adding all the spherical harmonics, is
\begin{eqnarray}
w(t,r,\vec \theta)&=&\sum_{l=0}^{l=\infty} \sum_{m=0}^{m=N(l,n)} w^{m}_{l}(r,t) 
Y_{lm}^n(\vec \theta) \nonumber \\Ê&=&
\sum_{l=0}^{l=\infty} \sum_{m=0}^{m=N(l,n)} \left[ 
u^{m}_{l}(t,r)-\int_0^r K^n_{lm}(r,\rho) u^{m}_{l}(t,\rho) d\rho,
\right]
Y_{lm}^n(\vec \theta)
\nonumber \\Ê&=&
u(t,r,\vec \theta)
-\int_0^r
 \int_{0}^{\pi} \hdots   \int_0^{\pi} \int_0^{2\pi} K(r,\rho,\vec \theta, \vec \phi) 
u (t,\rho,\vec \phi)
\sin^{n-2} \phi_{n-1}  \hdots \sin  \phi_2 d \vec \phi,\quad \label{eqn-transf-real}
\end{eqnarray}
where we have used (\ref{eqn-harm-coef}). We have defined $K$ as a function of its harmonics $ K^n_{lm}$,
\begin{equation}\label{eqn-K-harm-coef}
 K(r,\rho,\vec \theta, \vec \phi) =  \sum_{l=0}^{l=\infty} \sum_{m=0}^{m=N(l,n)}  K^n_{lm}(r,\rho) Y^n_{lm}(\vec \theta) \bar Y_{lm}^n(\vec \phi)
\end{equation}
Using (\ref{eqn-Knformula}) in (\ref{eqn-K-harm-coef}), we obtain $K$ as
\begin{eqnarray}
 K &=&-\rho \frac{\lambda}{\epsilon} \frac{\mathrm{I}_1\left[\sqrt{\frac{\lambda}{\epsilon}(r^2-\rho^2)}\right]}{\sqrt{\frac{\lambda}{\epsilon}(r^2-\rho^2)}}  
 \label{eqn-compl}
 \sum_{l=0}^{\infty} \sum_{m=0}^{N(l,n)}  \left(\frac{\rho}{r}\right)^{l+n-2} Y^n_{lm}(\vec \theta) \bar Y_{lm}^n(\vec \phi).    \end{eqnarray}
 The inner sum of (\ref{eqn-compl}) can be computed by virtue of the Addition Theorem for Spherical Harmonics~\cite[p.21]{atkinson}:
\begin{equation}
\sum_{m=0}^{N(l,n)} Y^n_{lm}(\vec \theta) \bar Y_{lm}^n(\vec \phi)  =\frac{N(l,n)}{\mathrm{Area}(S^{n-1})} P_{l,n}(\cos \omega),
\end{equation}
where $P_{l,n}$ is the Legendre polynomial of degree $l$ in $n$ dimensions, $\mathrm{Area}(S^{n-1})$ is the surface area of the $(n-1)$-sphere (given by (\ref{eqn-surf})), and $\omega$ represents the geodesic distance between the points given by the angles $(\vec \theta)$ and $(\vec \phi)$ on the $(n-1)$-sphere, given by: 
\begin{eqnarray}
\omega&=&\cos^{-1}\left\{
\cos \phi_{n-1}\cos \theta_{n-1}+\sin \phi_{n-1} \sin \theta_{n-1}
\times \left[ \cos \phi_{n-2}\cos \theta_{n-2}+\sin \phi_{n-2} \sin \theta_{n-2}
\right.\right.\nonumber \\ && \times\left.\left.
\left[\hdots \left[\cos \phi_2 \cos \theta_2+\sin \phi_2 \sin \theta_2 \cos (\theta_1-\phi_1)\right]\hdots \right]\right]\right\}.\quad \nonumber
\end{eqnarray}
Therefore, the nested sum of (\ref{eqn-compl}) is reduced to a single sum as follows
\begin{eqnarray}
 \sum_{l=0}^{\infty} \sum_{m=0}^{N(l,n)}  \left(\frac{\rho}{r}\right)^{l+n-2} Y^n_{lm}(\vec\theta) \bar Y_{lm}^n(\vec\phi) 
  &=& \sum_{l=0}^{\infty}   \left(\frac{\rho}{r}\right)^{l+n-2} 
\frac{N(l,n)}{\mathrm{Area}(S^{n-1})} P_{l,n}(\cos \omega).
\end{eqnarray}
This last sum can also be computed by using the Poisson identity~\cite[p. 54]{atkinson}, which is given by
\begin{equation}
\sum_{l=0}^{\infty} N(l,n) s^{l}  P_{l,n}(t)=\frac{1-s^2}{\left(1+s^2-2st\right)^{n/2}}.
\end{equation}
Therefore
\begin{eqnarray}
 \sum_{l=0}^{l}   \left(\frac{\rho}{r}\right)^{l+n-2} 
\frac{N(l,n)}{\mathrm{Area}(S^{n-1})} P_{l,n}(\cos \omega)
&=&\frac{1}{\mathrm{Area}(S^{n-1})} \left(\frac{\rho}{r}\right)^{n-2} \frac{1-\left(\frac{\rho}{r}\right)^2}{\left(\left(\frac{\rho}{r}\right)^2-2\left(\frac{\rho}{r}\right)\cos\omega+1\right)^{n/2}} \nonumber \\
&=&\frac{\rho^{n-2}}{\mathrm{Area}(S^{n-1})}  \frac{r^2-\rho^2}{\left(r^2+\rho^2-2\rho r \cos\omega\right)^{n/2}},
\end{eqnarray}
which is a function closely related to the Poisson kernel for the $n$-ball (see~\cite{evans},p.41, a function that tends to a Dirac delta $\delta(\theta-\psi)$ when $r$ goes to $\rho$, and helps solve Laplace's equation).
Using the sum of the series in (\ref{eqn-compl}) the control kernel is
\begin{eqnarray}
 K(r,\rho,\vec\theta,\vec \phi) 
&=& 
\frac{-\rho^{n-1}}{\mathrm{Area}(S^{n-1})}\sqrt{\frac{\lambda}{\epsilon}} \mathrm{I}_1\left[\sqrt{\frac{\lambda}{\epsilon}(r^2-\rho^2)}\right]  \label{eqn-k-real}
\frac{\sqrt{r^2-\rho^2}}{\left(r^2+\rho^2-2\rho r \cos\omega\right)^{n/2}}. \quad \end{eqnarray}
The kernel and the transformation can be written in rectangular coordinates, by defining $\vec \xi$ from radius $\rho$ and ultraspherical coordinates $\vec \phi$. Then,
 the transformation (\ref{eqn-transf-real}) is written as
\begin{equation}
w(t,\vec x)=\mathcal{K}[u]=
u(t,\vec x)
-  \int_{B^n(\Vert \vec x\Vert)}K(\vec x, \vec \xi) 
  u (t,\vec \xi)
d\vec\xi,\label{eqn-transf-real-rect}
\end{equation}where 
the symbol $\mathcal{K}[u]$ will be used to refer to the transformation in a simple way.
The integral in (\ref{eqn-transf-real-rect}) is extended over the $n$-ball of radius $\Vert \vec x \Vert$,
and noting that $r^2= \vec x\cdot \vec x$, $\rho^2= \vec \xi \cdot \vec \xi$, $r\rho\cos \omega= \vec x \cdot \vec \xi$, and that $\rho^{n-1}$ is part of the volume element in the integral, the kernel is
\begin{equation}
K(\vec x, \vec \xi)=
- \frac{\sqrt{\frac{\lambda}{\epsilon}}   \mathrm{I}_1\left[\sqrt{\frac{\lambda}{\epsilon}(\Vert \vec x \Vert ^2-\Vert \vec\xi \Vert ^2)}\right]  }{\mathrm{Area}(S^{n-1}) }
 \frac{\sqrt{\Vert \vec x \Vert ^2-\Vert \vec \xi \Vert ^2}}{\Vert \vec x - \vec \xi \Vert^n}, \quad \label{control-kernel-rect}
 \end{equation}
and the control law (\ref{control}) is found by inserting (\ref{control-kernel-rect}) into (\ref{eqn-transf-real-rect}) and fixing $\vec x \in S^{n-1}(R)$ (thus $\Vert \vec x \Vert=R$).

\section{Observer design}\label{sect-observer}
Consider now the problem of designing an observer for (\ref{eqn-ndim-spherical})--(\ref{eqn-ndim-spherical-obs}). Working in spherical harmonics, we start from (\ref{eqn-un})--(\ref{eqn-measurementun}) and construct our observer as a copy of the plant plus an output injection term:
\begin{eqnarray}
\hat u_{lmt}&=&\frac{\epsilon}{r^{n-1}} \partial_r \left( r^{n-1}\partial_r\hat  u_{lm} \right)-l(l+n-2) \frac{\epsilon}{r^2} \hat  u_{lm}+\lambda \hat  u_{lm}+p_{lm}^n(r) (y_{lm}(t)-\partial_r \hat u_{lm}(t,R)),
\end{eqnarray}
with  boundary conditions 
\begin{eqnarray}
\hat u_{lm}(t,R)&=& U_{lm}(t).
\end{eqnarray}
We need to design the output injection gain $p_{lm}^n(r)$. Define the observer error as $\tilde u=u-\hat u$. The observer error dynamics are given by
\begin{eqnarray}\label{eqn-utilden}
\tilde u_{lmt}&=&\frac{\epsilon}{r^{n-1}} \partial_r \left( r^{n-1}\partial_r\tilde  u_{lm} \right)-l(l+n-2) \frac{\epsilon}{r^2} \tilde  u_{lm}+\lambda \tilde  u_{lm}-p_{lm}^n(r) \partial_r\tilde u_{lm}(t,R),
\end{eqnarray}
with  boundary conditions 
\begin{eqnarray}
\tilde u_{lm}(t,R)&=&0.
\end{eqnarray}
 Next we use the backstepping method to find a value of $p_{lm}^n(r)$ that guarantees convergence of $\tilde u$ to zero. This ensures that the observer estimates tend to the true state values. Our approach to design $p_{lm}^n(r)$ is to seek a mapping that transforms (\ref{eqn-utilden}) into the following target system
\begin{eqnarray}\label{eqn-wtilden}
\tilde w_{lmt}&=&\frac{\epsilon}{r^{n-1}} \partial_r \left( r^{n-1} \partial_r \tilde w_{lm} \right)-l(l+n-2) \frac{\epsilon}{r^2} \tilde w_{lm},
\end{eqnarray}
a stable heat equation (a negative reaction coefficient could also be added if a more rapid observer convergence rate is desired),
with boundary conditions
\begin{eqnarray}
\tilde w_{lm}(t,R)&=&0.
\end{eqnarray}

The transformation is defined as follows:
\begin{equation}\label{eqn-obstrans}
\tilde u_{lm}(t,r)=\tilde w_{lm}(t,r)-\int_r^RP^n_{lm}(r,\rho) \tilde w_{lm}(t,\rho) d\rho
\end{equation}
and then $p_{lm}^n(r)$ will be found from transformation kernel as an additional condition.

To find the kernel equations,  as in Section~\ref{sec-design} one has substitute the transformation into the original system and substitute the target systems, integrating by parts when possible. It is straightforward to obtain the following PDE that the kernel must verify:
\begin{eqnarray}\label{obs-kernel}
\frac{1}{r^{n-1}} \partial_r \left(r^{n-1}   \partial_r P^n_{lm} \right)
-  \partial_\rho \left(\rho^{n-1} \partial_\rho \left( \frac{P^n_{lm}}{\rho^{n-1}} \right)\right)
-l(l+n-2) \left(\frac{1}{r^2}- \frac{ 1}{\rho^2} \right)P^n_{lm}=-\frac{\lambda}{\epsilon} P^n_{lm}
\end{eqnarray}
In addition we find a value for the output injection gain kernel
\begin{equation}
p_{lm}^n(r)=\epsilon P^n_{lm}(r,R)  
\end{equation}

Also, the following boundary condition has to be verified
\begin{eqnarray}
0&=& \lambda  + \partial_r \epsilon  P^n_{lm} (r,r)  +\frac{ \epsilon}{r^{n-1}}  \frac{\partial}{\partial r} \left(r^{n-1}   P^n_{lm}(r,r) \right) 
+\epsilon \partial_\rho
\left( \frac{P^n_{lm} (r,\rho)}{\rho^{n-1}}\right) \bigg|_{\rho=r} r^{n-1} ,
\end{eqnarray}
which can be written as
\begin{eqnarray}
0&=& \lambda  +  \epsilon  \partial_r P^n_{lm} (r,r) +\epsilon \frac{d}{dr}\left(P^n_{lm}(r,r)\right) +(n-1)\frac{ \epsilon P^n_{lm}(r,r)  }{r}  +\epsilon \partial_\rho P^n_{lm}(r,r) - (n-1)\frac{ \epsilon P^n_{lm}(r,r)  }{r}.\quad
\end{eqnarray}
Operating, we obtain
\begin{eqnarray}
0&=& \lambda  +  2 \epsilon \frac{d}{dr}\left(P^n_{lm}(r,r)\right),
\end{eqnarray}
which integrates to
\begin{eqnarray}
P^n_{lm}(r,r) 
&=& - \frac{\lambda r P^n_{lm}(\rho,\rho) }{2 \epsilon }+C.
\end{eqnarray}
To obtain the constant $C$ and another boundary condition, we note that the PDE verified by both $\tilde u$ and $\tilde w$ has to have meaningful (non-singular) values at $r\rightarrow 0$. This implies
\begin{eqnarray}
\lim_{r\rightarrow 0} \left[(n-1) \partial_r \tilde w_{lm} -l(l+n-2) \frac{\tilde w_{lm}}{r}\right]&=&0,\\
\lim_{r\rightarrow 0} \left[(n-1) \partial_r \tilde u_{lm} -l(l+n-2) \frac{\tilde u_{lm}}{r}\right]&=&0.
\end{eqnarray}
From the transformation:
\begin{eqnarray}
&&(n-1) \partial_r \tilde u_{lm} -l(l+n-2) \frac{\tilde u_{lm}}{r} \nonumber \\
&=&(n-1) \partial_r \tilde w_{lm} -l(l+n-2) \frac{\tilde w_{lm}}{r}
-\int_r^R \left[(n-1) \partial_r P^n_{lm}(r,\rho)-l(l+n-2) \frac{P^n_{lm}(r,\rho)}{r} \right] \tilde w_{lm}(\rho) d\rho
\nonumber \\
&&
+  (n-1) P^n_{lm}(r,r)\tilde w_{lm}(t,\rho),
\end{eqnarray}
thus it is sufficient that $P^n_{lm}(0,\rho)=\partial_r P^n_{lm}(0,\rho)=0$

Thus the boundary conditions for the kernel equations become
\begin{eqnarray}
P^n_{lm}(0,\rho)&=&\partial_r P^n_{lm} (0,\rho)=0,\\
P^n_{lm}(r,r) 
&=& -\frac{\lambda r  }{2 \epsilon }.
\end{eqnarray}
It turns out the observer kernel equation can be transformed into the control kernel equation, therefore obtaining a similar explicit result. For this, define
\begin{equation}
\check{P}^n_{lm}(r,\rho)=\frac{\rho^{n-1}}{r^{n-1}} P^n_{lm}(\rho,r),
\end{equation}
and it can be verified that (\ref{obs-kernel}) is transformed into (\ref{eqn-Kn}), along with the boundary conditions. Thus $\check{P}^n_{lm}(r,\rho)=K^n_{lm}(r,\rho)$ and we find
\begin{equation}
{P}^n_{lm}(r,\rho)=\frac{r^{n-1}}{\rho^{n-1}} K^n_{lm}(\rho,r)=-\rho\left(\frac{\rho}{r}\right)^{l-1}  \frac{\lambda}{\epsilon} \frac{\mathrm{I}_1\left[\sqrt{\frac{\lambda}{\epsilon}(\rho^2-r^2)}\right]}{\sqrt{\frac{\lambda}{\epsilon}(\rho^2-r^2)}}.
\end{equation}
Summing as before all spherical harmonics leads to a  transformation kernel 
\begin{eqnarray}
P(r,\rho,\vec\theta,\vec \phi) 
&=& 
\frac{-\rho^{n-1}}{\mathrm{Area}(S^{n-1})}\sqrt{\frac{\lambda}{\epsilon}} \mathrm{I}_1\left[\sqrt{\frac{\lambda}{\epsilon}(\rho^2-r^2)}\right]  \label{eqn-p-real}
\frac{\sqrt{\rho^2-r^2}}{\left(r^2+\rho^2-2\rho r \cos\omega\right)^{n/2}}, \quad \end{eqnarray}
which would define the observer transformation in physical space as
\begin{equation}
\tilde u(t,\vec x)=\mathcal{P}[\tilde w]=
\tilde w(t,\vec x)
-  \int_{B^n(R)- B^n(\Vert \vec x\Vert)}P(\vec x, \vec \xi) 
  \tilde w (t,\vec \xi)
d\vec\xi,\label{eqn-obstransf-real-rect}
\end{equation}
where the integral in (\ref{eqn-obstransf-real-rect}) is extended to the $n$-ball from radius $R$ to radius $\Vert \vec x \Vert$.
Finally, using $p_{lm}^n(r)=\epsilon P^n_{lm}(r,R)$ and summing the spherical harmonics yields the physical-space operator in the right-hand side of (\ref{eqn-ndim-obs-rect-th}).

\section{Proof of closed-loop stability in the $H^1$ norm} \label{sect-stability}
We first remark that studying the augmented ($u,\hat u$) system is equivalent to studying the augmented ($\tilde u,\hat u$) system. 
To obtain the stability result of Theorem~\ref{th-main2} we  need three elements. We begin by obtaining the existence of an inverse transformation (for both control and observer transformations) that allows us to recover the original variable from the transformed variable. Then we relate the $H^1$ norm with spherical harmonics and show that both transformations are invertible maps from $H^1$ into $H^1$ (Proposition~\ref{prop-hnorms}). We continue by stating a well-posedness and stability result for the augmented ($\tilde w,\hat w$) system in physical space (Proposition~\ref{targetsystem}). Combining the two propositions, it is straightforward to construct the proof of Theorem~\ref{th-main2} by mapping the results for the target augmented system to the original augmented system.

\subsection{Invertibility of the transformations}
We start with the control transformation. Mimicking (\ref{eqn-transf}), we start by posing an inverse transform in the spherical harmonics space, as follows
\begin{equation}\label{eqn-invtransf}
u_{l}^{m}(t,r)=w_{l}^{m}(t,r)+\int_0^r L^n_{lm}(r,\rho) w_{l}^{m}(t,\rho) d\rho,
\end{equation}
and proceeding in the same fashion of Section~\ref{sect-back} we find the following kernel equations for $L^n_{lm}$:
\begin{eqnarray}
&&\frac{1}{r^{n-1}} \partial_r \left(r^{n-1}  \partial_r  L^n_{lm} \right)
-  \partial_\rho  \left(\rho^{n-1} \partial_\rho  \left( \frac{L^n_{lm}}{\rho^{n-1}} \right) \right)
-l(l+n-2) \left(\frac{1}{r^2}- \frac{ 1}{\rho^2} \right)L^n_{lm}=-\frac{\lambda}{\epsilon} L^n_{lm}.\label{eqn-Ln}
\end{eqnarray} 
with boundary conditions
\begin{eqnarray}
L^n_{lm}(r,0)&=&(n-2) \partial_\rho  L^n_{lm} (r,0)=0,\\
L^n_{lm}(r,r) 
&=&-\frac{\lambda r }{2 \epsilon }. \label{eqn-lnbc}
\end{eqnarray}
These equations are very similar to (\ref{eqn-Kn}) but substituting $\lambda$ by $-\lambda$ and changing the sign of the kernel. Thus we easily find the inverse transform by doing these substitutions in (\ref{eqn-Knformula}) as
\begin{equation}
L^n_{lm}(r,\rho)=-\rho\left(\frac{\rho}{r}\right)^{l+n-2} \frac{\lambda}{\epsilon} \frac{\mathrm{J}_1\left[\sqrt{\frac{\lambda}{\epsilon}(r^2-\rho^2)}\right]}{\sqrt{\frac{\lambda}{\epsilon}(r^2-\rho^2)}},\label{eqn-klnexpl}
\end{equation}
where $\mathrm{J}_1$ is the first-order Bessel function of the first kind.

Then, as in (\ref{eqn-transf-real-rect}), the transformation in physical space is
\begin{equation}
u(t,\vec x)=\mathcal{L}[ w]=
w(t,\vec x)
+  \int_{B^n(\Vert \vec x\Vert)}L(\vec x, \vec \xi) 
  w (t,\vec \xi)
d\vec\xi,\label{eqn-transf-inv-real-rect}
\end{equation}where the kernel $L$ is
\begin{equation}
L(\vec x, \vec \xi)=
-\frac{\sqrt{\frac{\lambda}{\epsilon}}   \mathrm{J}_1\left[\sqrt{\frac{\lambda}{\epsilon}(\Vert \vec x \Vert ^2-\Vert \vec\xi \Vert ^2)}\right]  }{\mathrm{Area}(S^{n-1}) }
 \frac{\sqrt{\Vert \vec x \Vert ^2-\Vert \vec \xi \Vert ^2}}{\Vert \vec x - \vec \xi \Vert^n}. \quad \label{control-invkernel-rect}
 \end{equation}
It is obvious that a very similar result can be achieved for the observer transformation, which we will define
\begin{equation}
\tilde w(t,\vec x)=\mathcal{R}[\tilde u]=
\tilde u(t,\vec x)
+  \int_{B^n(R)- B^n(\Vert \vec x\Vert)}R(\vec x, \vec \xi) 
  \tilde u (t,\vec \xi)
d\vec\xi,\label{eqn-obstransf-real-rect-inv}
\end{equation}
with the kernel $R$ being very similar in structure to $L$.
\subsection{Computing  norms using spherical harmonics}

The $L^2$ norm (\ref{eqn-defL2}) written in ultraspherical coordinates is
\begin{eqnarray}
\Vert f \Vert_{L^2}&=&\left( 
\int_0^R
 \int_{0}^{\pi} \hdots   \int_0^{\pi} \int_0^{2\pi} 
 f^2 (r,\vec \theta) 
\sin^{n-2} \theta_{n-1}  \hdots \sin  \theta_2 r^{n-1} d \vec \phi dr,
\right)^{1/2},\quad
\end{eqnarray}
and using the properties of spherical harmonics, if $f^{m}_{l}(r)$ are the harmonics of $f$,
then
\begin{eqnarray} \label{eqn-l2sph}
\Vert f \Vert_{L^2}&=&\left( 
\sum_{l=0}^{l=\infty} \sum_{m=0}^{m=N(l,n)}
\int_0^R
 \vert f^{m}_{l}(r)  \vert^2
r^{n-1} dr,
\right)^{1/2}.\qquad
\end{eqnarray}

Similarly, the $H^1$ norm (\ref{eqn-defL2}) written in ultraspherical coordinates is
\begin{eqnarray}
\Vert f \Vert_{H^1}^2&=&\Vert f \Vert_{L^2}^2+
\int_0^R
 \int_{0}^{\pi} \hdots   \int_0^{\pi} \int_0^{2\pi} 
\Vert \vec \nabla f(r,\vec \theta) \Vert^2 
\sin^{n-2} \theta_{n-1}  \hdots \sin  \theta_2 r^{n-1} d \vec \phi dr,
\quad
\end{eqnarray}
Now, noting that 
\begin{equation}
\vec \nabla f(r,\vec \theta)=\vec \xi \, \frac{\partial f}{\partial r}+\frac{1}{r} \vec \nabla^*_{n-1} f
\end{equation}
where (see~\cite[p. 90]{atkinson}) $\vec \xi$ is an unitary vector pointing in the radial direction at $(r,\vec \theta)$ and $\vec \nabla^*_{n-1}$ is the first-order Beltrami operator on the $n-1$ unitary sphere, we obtain
\begin{equation}
\Vert \vec \nabla f(r,\vec \theta) \Vert^2 =\left(\frac{\partial f}{\partial r}\right)^2+\frac{1}{r^2} \vec \nabla^*_{n-1} f \cdot  \vec \nabla^*_{n-1} f
\end{equation}
and therefore 
\begin{equation}
\Vert f \Vert_{H^1}^2=\Vert f \Vert_{L^2}+
\Vert \frac{\partial f}{\partial r} \Vert^2_{L^2} +
\int_0^R \frac{1}{r^2}  \left[
 \int_{0}^{\pi} \hdots   \int_0^{\pi} \int_0^{2\pi} 
\vec \nabla^*_{n-1} f \cdot  \vec \nabla^*_{n-1} f 
\sin^{n-2} \theta_{n-1}  \hdots \sin  \theta_2  d \vec \phi \right] r^{n-1} dr,
\end{equation}
and considering the expansion of $f$ in spherical harmonics, and the Green-Beltrami identity~\cite[p. 95]{atkinson}, we obtain
\begin{equation} \label{eqn-h1sph}
\Vert f \Vert_{H^1}^2=
\sum_{l=0}^{l=\infty} \sum_{m=0}^{m=N(l,n)} 
\int_0^R \left[\left(1+\frac{l(l+n-2)}{r^2}\right)
 \vert f^{m}_{l}(r)  \vert^2 +  \left|\frac{\partial f^{m}_{l}}{\partial r}  \right|^2
  \right]
r^{n-1} dr
.
\end{equation}

\subsection{The control and observer transformation as maps between functional spaces}
We next show that both the direct and inverse control and observer transformation transform $L^2$ (resp. $H^1$) functions back into $L^2$ (resp. $H^1$) functions.
\begin{proposition}\label{prop-hnorms}
Assume that the function $g(r,\vec x)$ is related to the function $f(r,\vec x)$ by means of the transformation $g=\mathcal{K}[f]$ and consequently $f$ is related to $g$ by $f=\mathcal{L}[g]$. Then:
\begin{equation}
\Vert g \Vert_{L^2} \leq C_{KL} \Vert f \Vert_{L^2},\, 
\Vert f \Vert_{L^2} \leq C_{LL} \Vert g \Vert_{L^2},\, 
\Vert g \Vert_{H^1} \leq C_{KH} \Vert f \Vert_{H^1},\, 
\Vert f \Vert_{H^1} \leq C_{LH} \Vert g \Vert_{H^1},\, 
\end{equation}
where constants $C_{KL},C_{LL},C_{KH},C_{LH}$ depending only on $R$, $\lambda$, $\epsilon$,  and $n$. 

Similarly, assume that the function $g(r,\vec x)$ is related to the function $f(r,\vec x)$ by means of the observer transformation $g=\mathcal{P}[f]$ and consequently $f$ is related to $g$ by $f=\mathcal{R}[g]$. Then:
\begin{equation}
\Vert g \Vert_{L^2} \leq C_{PL} \Vert f \Vert_{L^2},\, 
\Vert f \Vert_{L^2} \leq C_{RL} \Vert g \Vert_{L^2},\, 
\Vert g \Vert_{H^1} \leq C_{PH} \Vert f \Vert_{H^1},\, 
\Vert f \Vert_{H^1} \leq C_{RH} \Vert g \Vert_{H^1},\, 
\end{equation}
where constants $C_{PL},C_{RL},C_{PH},C_{RH}$ depending only on $R$, $\lambda$, $\epsilon$,  and $n$.
\end{proposition}
\begin{proof}
We show the result for the direct control transformation $\mathcal{K}$. The inverse transformation shares essentially the same structure and therefore the same proof applies.

First, note that, for $0\leq \rho \leq r \leq R$, and since $\frac{\mathrm{I}_1\left[x\right]}{x}$ is a continuous and differentiable function for $x\geq 0$, we have
\begin{eqnarray}
\frac{\lambda}{\epsilon} \frac{\mathrm{I}_1\left[\sqrt{\frac{\lambda}{\epsilon}(r^2-\rho^2)}\right]}{\sqrt{\frac{\lambda}{\epsilon}(r^2-\rho^2)}} \leq C_I, \quad \frac{\partial}{\partial r} \left( \frac{\lambda}{\epsilon} \frac{\mathrm{I}_1\left[\sqrt{\frac{\lambda}{\epsilon}(r^2-\rho^2)}\right]}{\sqrt{\frac{\lambda}{\epsilon}(r^2-\rho^2)}} \right) \leq C_J
\end{eqnarray}
where $C_I$ and $C_J$ are functions of $R$, $\lambda$ and $\epsilon$.

We work in spherical harmonics, for which the transformation is defined as (\ref{eqn-transf}). Then we have
\begin{eqnarray}
\vert g^{m}_{l} \vert^2 &=& \left|
f^{m}_{l}(r)-\int_0^r K_{lm}^n(r,\rho) f^{m}_{l}(\rho) d\rho
\right|^{2} 
\nonumber \\ 
&\leq&
2 \vert f^{m}_{l} \vert^2
+ 2 \left|
\int_0^r  K_{lm}^n(r,\rho) f^{m}_{l}(\rho) d\rho \right|^{2}
\nonumber \\ 
&\leq&
2 \vert f^{m}_{l} \vert^2
+ 2 C_I^2 \left(
\int_0^r  \rho \left( \frac{\rho}{r} \right)^{l+n-2}  \vert f^{m}_{l}(\rho) \vert d\rho \right)^{2}
\nonumber \\ 
&\leq&
2 \vert f^{m}_{l} \vert^2
+ 2 C_I^2  \left(
\int_0^r  \rho \left( \frac{\rho}{r} \right)^{l+n-2}  d\rho \right)
 \left(
\int_0^r  \rho \left( \frac{\rho}{r} \right)^{l+n-2}  \vert f^{m}_{l}(\rho) \vert^2d\rho \right)
\nonumber \\ 
&=&
2 \vert f^{m}_{l} \vert^2
+  \frac{2 C_I^2r^2}{l+n}
 \left(
\int_0^r  \rho \left( \frac{\rho}{r} \right)^{l+n-2}  \vert f^{m}_{l}(\rho) \vert^2d\rho \right)
\nonumber \\ 
&\leq&
2 \vert f^{m}_{l} \vert^2
+   \frac{2 C_I^2r^{4-n}}{n} \left(
\int_0^R \rho^{n-1} \vert f^{m}_{l}(\rho) \vert^2d\rho \right),\,
\end{eqnarray}
and therefore, applying (\ref{eqn-l2sph}),
\begin{eqnarray}
\Vert g \Vert_{L^2}^2
&\leq & \left(2+ \frac{ R^{4} C_I^2}{2n}\right) \Vert f \Vert_{L^2}^2= C_{KL} \Vert f \Vert_{L^2}^2.\quad
\end{eqnarray}
This shows the $L^2$ part of the proposition. To prove the $H^1$ part, note that
\begin{equation}
\frac{\partial g^{m}_{l}}{\partial r} =
\frac{\partial f^{m}_{l}}{\partial r} - K_{lm}^n(r,r) f^{m}_{l}(r) -\int_0^r \frac{\partial K_{lm}^n(r,\rho)}{\partial r} f^{m}_{l}(\rho) d\rho
\end{equation}
Thus
\begin{equation}
\left| \frac{\partial g^{m}_{l}}{\partial r} \right|^2 
\leq 
3 \left| \frac{\partial f^{m}_{l}}{\partial r} \right|^2
+3 C_I^2 \left|  f^{m}_{l} \right|^2
+   \frac{3 C_J^2r^{4-n}}{n} \left(
\int_0^R \rho^{n-1} \vert f^{m}_{l}(\rho) \vert^2d\rho \right),\,
\end{equation}
and we obtain from (\ref{eqn-h1sph})
\begin{eqnarray} 
\Vert g \Vert_{H^1}^2&\leq&
\sum_{l=0}^{l=\infty} \sum_{m=0}^{m=N(l,n)} 
\int_0^R \left[\left(1+\frac{l(l+n-2)}{r^2}\right)
\left[2 \vert f^{m}_{l} \vert^2
+   \frac{2 C_I^2r^{4-n}}{n} \left(
\int_0^R \rho^{n-1} \vert f^{m}_{l}(\rho) \vert^2d\rho \right)\right] 
\right. \nonumber \\ && 
+ \left. 3 \left| \frac{\partial f^{m}_{l}}{\partial r} \right|^2
+3 C_I^2 \left|  f^{m}_{l} \right|^2
+   \frac{2 C_J^2r^{4-n}}{n} \left(
\int_0^R \rho^{n-1} \vert f^{m}_{l}(\rho) \vert^2d\rho \right)
  \right]
r^{n-1} dr
\nonumber \\ &\leq & \left(3+ \frac{C_I^2R^4}{n} \right) \Vert f \Vert_{H^1}^2+\left(2+ 3 \frac{ R^{4}C_J^2}{4n}\right)\Vert f \Vert_{L^2}^2
 \nonumber \\ &=&  
C_{KH} \Vert f \Vert_{H^1}^2,
\end{eqnarray}
concluding the proof.
\end{proof}

\subsection{Stability of the target system}
Consider first the ($\tilde w,\hat w$) system  with control law (\ref{control-obs}), where $\hat w=\mathcal{K}[\hat u]$ and $\tilde w$ is defined by $\tilde w(t,\vec x)=\mathcal{P}[\tilde u]$.
The PDEs verified by ($\tilde w,\hat w$) are
\begin{eqnarray}\label{eqn-w}
\partial_t \tilde w&=&\epsilon \bigtriangleup_n \tilde w,\\ \label{eqn-what}
\partial_t \hat w&=&\epsilon \bigtriangleup_n \hat w-\mathcal{F}[\tilde w_r(t,\vec x)]
\end{eqnarray}
with boundary conditions
\begin{eqnarray}
\tilde w(t,\vec x)\Big|_{\vec x\in S^{n-1}(R)}&=&0,\label{eqn-bcw1} \\
\hat w(t,\vec x)\Big|_{\vec x\in S^{n-1}(R)}&=&0,\label{eqn-bcw2}
\end{eqnarray}
where $\mathcal{F}[\tilde w_r(t,\vec x)]$
is
\begin{equation}
\mathcal{F}[\tilde w_r(t,\vec x)]=\mathcal{K}\left[
\frac{\sqrt{ \lambda \epsilon} }{\mathrm{Area}(S^{n-1}) }   \mathrm{I}_1\left[\sqrt{\frac{\lambda}{\epsilon}(R^2-\Vert \vec x \Vert ^2)} \right]  \sqrt{R^2-\Vert \vec x \Vert ^2}  \int_{S^{n-1}(R)}  \frac{\tilde w_r(t,\vec \xi)}{\Vert \vec x - \vec \xi \Vert^n}
d\vec \xi\right],
\end{equation}
and with $\hat w_0=\mathcal{K}[\hat u_0]=0$, $\tilde w_0=\mathcal{P}[\tilde u_0]\in H^1_0$. Notice that the PDE system is actually a \emph{cascade} system; $\tilde w$ verifies an autonomous PDE and its solution (or more specifically, the integral of a certain trace of the solution on the boundary) drives the PDE for $\hat w$. The following result holds
\begin{proposition}\label{targetsystem}
Consider the system (\ref{eqn-w})--(\ref{eqn-bcw2}) with initial conditions $\tilde w_0\in H^1_0$, $\tilde w_0=0$ . Then, $\tilde w,\hat w \in{\cal C}\left[[0,\infty),H^1_0\right]\cap L^2\left[(0,\infty),H^2 \right]$  and also $\partial_t \tilde w, \partial_t \hat w \in L^2\left[(0,\infty),L^2 \right]$. Moreover, the following bounds are verified
\begin{eqnarray} \label{eqn-result}
\Vert \tilde w(t,\cdot) \Vert_{H^1} &\leq& D_1 \mathrm{e}^{-\alpha_1 t} \Vert \tilde w_0 \Vert_{H^1},\\
\Vert \tilde w(t,\cdot) \Vert_{H^1}+\Vert \hat w(t,\cdot) \Vert_{H^1} &\leq& D_2 \mathrm{e}^{-\alpha_2 t} \Vert \tilde w_0 \Vert_{H^1}, \label{eqn-result2}
\end{eqnarray}
where $D_1,D_2,\alpha_1,\alpha_2$ are positive constants.
\end{proposition}
The well-posedness part of the result is standard for the $\tilde w$ system (see for instance~\cite{brezis}, p. 328). For the $\hat w$ system, notice that, given the regularity of $\tilde w$, the trace of $\tilde w_r$ on the sphere is an $L^2$ function. Since $\mathcal{F}[\tilde w_r(t,\vec x)]$ basically amounts to the composition of the observer and controller transformation, which map $L^2$ into $L^2$, we obtain the same regularity result (see for instance~\cite{brezis}, p. 357).

The stability estimate is obtained using a Lyapunov argument. Define  $V_1(t)=\frac{1}{2} \Vert \tilde w(t,\cdot) \Vert_{L^2}^2$ and $V_2(t)=\frac{1}{2} \Vert \vec \nabla_n \tilde w(t,\vec x)\Vert_{L^2}^2$. Then, first
\begin{equation}
\dot V_1=\epsilon \int_{B^n(R)}w(t,\vec x)\bigtriangleup_n \tilde w(t,\vec x)
d\vec x.
\end{equation}
Applying one of Green's formulas (see~\cite{evans}, p.628), and since the normal on the boundary of the $B^n(R)$ is just $\frac{\vec x} {R}$
\begin{eqnarray}
\dot V_1&=&-\epsilon \int_{B^n(R)} \Vert \vec \nabla_n \tilde w(t,\vec x)\Vert^2
d\vec x
+\int_{S^{n-1}(R)} \tilde w(t,\vec x) \frac{\vec \nabla_n \tilde w(t,\vec x) \cdot \vec x}{R} d\vec x,\label{eqn-green} 
\end{eqnarray}
where $\vec \nabla_n$ is the gradient operator in $n$-dimensional space. The second integral of (\ref{eqn-green}) is found to be zero by applying (\ref{eqn-bcw1}). To bound the first integral, we use a Poincar\'e-type inequality (see for instance~\cite{brezis}, page 290),
\begin{equation}
\int_{B^n(R)}  w^2(t,\vec x)
d\vec x
\leq C_p
\int_{B^n(R)} \Vert \vec \nabla_n w(t,\vec x)\Vert^2
d\vec x,
\end{equation}
which also implies $V_1 \leq C_p V_2$.
thus we reach $\dot V_1=-2 \epsilon V_2 \leq -2C_p\epsilon V_1$. On the other hand,
\begin{equation}
\dot V_2=\epsilon \int_{B^n(R)} \vec \nabla_n \tilde w \cdot  \vec \nabla_n \tilde w_t
d\vec x.
\end{equation}
Applying again one of Green's formulas, we obtain
\begin{eqnarray}
\dot V_2=- \epsilon \int_{B^n(R)} (\bigtriangleup_n \tilde w(t,\vec x))^2
d\vec x=- \epsilon \Vert \bigtriangleup_n \tilde w \Vert_{L^2}^2,
\end{eqnarray}
and therefore we obtain that
$$
\dot V_1+\dot V_2 \leq -\epsilon(1+C_p) (V_1+V_2),
$$
and applying Gronwall's Inequality we obtain the stability result (\ref{eqn-result}). To obtain (\ref{eqn-result2}), define now $V_3(t)=\frac{1}{2} \Vert \hat w(t,\cdot) \Vert_{L^2}^2$ and $V_4(t)=\frac{1}{2} \Vert \vec \nabla_n \hat w(t,\vec x)\Vert_{L^2}^2$. We obtain the same results as before with additional terms due to the forcing function in (\ref{eqn-what}), namely
\begin{eqnarray}
\dot V_3&=&-2 \epsilon V_4 + \int_{B^n(R)}\hat w \mathcal{F}[\tilde w_r(t,\vec x)] d\vec x,\\
\dot V_4&=&- \epsilon \int_{B^n(R)} (\bigtriangleup_n \hat w(t,\vec x))^2 + \int_{B^n(R)}\bigtriangleup_n \hat w \mathcal{F}[\tilde w_r(t,\vec x)] d\vec x
\end{eqnarray}
which can be bounded as
\begin{eqnarray}
\dot V_3&\leq&-\frac{\epsilon(1+C_p)}{2}(V_3 + V_4)+\frac{1}{2\epsilon C_p}\int_{B^n(R)}\left( \mathcal{F}[\tilde w_r(t,\vec x)]\right)^2 d\vec x,\\
\dot V_4&\leq &\frac{1}{4\epsilon}\int_{B^n(R)}\left( \mathcal{F}[\tilde w_r(t,\vec x)]\right)^2 d\vec x
\end{eqnarray}
On the other hand, 
\begin{eqnarray}
\int_{B^n(R)}\left( \mathcal{F}[\tilde w_r(t,\vec x)]\right)^2 d\vec x \leq K_S \int_{S^{n-1}(R)}\tilde w_r^2(t,\vec x) d\vec x,
\end{eqnarray}
and by virtue of the trace theorem~\cite[p. 258]{evans}
\begin{eqnarray}
\int_{S^{n-1}(R)}\tilde w_r^2(t,\vec x) d\vec x
\leq K_T  \Vert \tilde w_r \Vert_{H^1}^2 
\end{eqnarray}
and since $\tilde w$ vanishes at the boundary $\Vert \tilde w_r \Vert_{H^1}^2 \leq K_V \Vert \bigtriangleup_n \tilde w \Vert_{L^2}^2$. Thus, we reach
\begin{eqnarray}
\dot V_3+\dot V_4&\leq&-\frac{\epsilon(1+C_p)}{2}(V_3 + V_4)+K_F \Vert \bigtriangleup_n \tilde w \Vert_{L^2}^2,
\end{eqnarray}
for $K_F$ positive. Then, defining $V_5=V_1+V_2+\frac{\epsilon}{K_F} (V_3+V_4)$, we obtain $\dot V_5\leq -K_W V_5$, and applying Gronwall's Inequality and taking into account $\hat w_0=0$ we obtain the final result (\ref{eqn-result2}).
\subsection{Proof of Theorem~\ref{th-main2}}
Using Proposition~\ref{prop-hnorms} and Proposition~\ref{targetsystem}, we directly obtain the well-posedness result of Theorem~\ref{th-main2}. Finally, the stability result is found as follows
\begin{eqnarray}
\Vert u(t,\cdot) \Vert_{H^1}+\Vert \hat u(t,\cdot) \Vert_{H^1} &\leq& C_{PH}
\Vert \tilde w (t,\cdot) \Vert_{H^1}+( C_{PH}+C_{KH}) \Vert \hat w(t,\cdot) \Vert_{H^1}
\nonumber \\
&\leq &
( C_{PH}+C_{KH}) D_2 \mathrm{e}^{-\alpha_2 t} \Vert \tilde w_0 \Vert_{H^1}
\nonumber \\
&\leq &
C_{RH} ( C_{PH}+C_{KH}) D_2 \mathrm{e}^{-\alpha_2 t} \Vert u_0 \Vert_{H^1},
\end{eqnarray}
thus concluding the proof.

\section{Conclusion}\label{sec-conclusions}
We have shown an explicit output-feedback design to stabilize a constant-coefficient reaction-diffusion equation on an $n$-ball, by using a boundary output-feedback control law found using backstepping. The resulting control law and observer injection gains have a remarkable structure. They are formulated as a multiple integral whose kernel is a product of two factors, the first of which is identical to the backstepping kernel used in control of one-dimensional reaction-diffusion equations. The second factor is closely related to the Poisson kernel in the $n$-ball (a function used to solve Laplace's problem). While an $H^1$ stability result has been achieved, future work includes analysing well-posedness and stability in higher Sobolev spaces to improve the regularity of the closed-loop solutions. Other open problems include extending the result to spatially- or time-varying coefficients, or to more complicated domains.

Finally, recent applications of PDE control theory in the field of multi-agent systems suggest that a possible application of this $n$-dimensional result might lie in control of complex, large dimensional  systems, such as social networks and opinion dynamics.


\begin{thebibliography}{xx}


\bibitem{abramowitz}
M.~Abramowitz and I.~A.~Stegun, {\em Handbook of mathematical functions},  9th Edition, Dover, 1965.


\bibitem{atkinson}
K.Atkinson and W.~Han, {\em Spherical Harmonics and Approximations on the Unit Sphere: An Introduction}, Springer, 2012.

\bibitem{orlov}
A.~Baccoli, A.~Pisano, Y.~Orlov, ``Boundary control of coupled reactionÐdiffusion processes with constant parameters,''
{\em Automatica}, vol. 54, pp. 80--90, 2015.

\bibitem{Barbu}
V.~Barbu, ``Boundary Stabilization of Equilibrium Solutions to Parabolic Equations,'' {\em IEEE Transactions on Automatic Control}, vol. 58, pp. 2416--2420, 2013.

\bibitem{brezis}
H.~Brezis,  {\em Functional analysis, Sobolev spaces and Partial Differential Equations}, Springer, 2011.

\bibitem{Prieur}
F.~Bribiesca~Argomedo, C.~Prieur, E.~Witrant, and S.~Bremond, ``A Strict Control Lyapunov Function for a Diffusion Equation With Time-Varying Distributed Coefficients,'' {\em IEEE Trans. Autom. Control}, vol. 58, pp. 290--303, 2013.


\bibitem{vazquez-nonlinear}
J.-M.~Coron, R.~Vazquez, M.~Krstic, and G.~Bastin, ``Local Exponential ${H^2}$ Stabilization of a ${2\times 2}$ Quasilinear Hyperbolic System using Backstepping,'' {\em SIAM J. Control Optim.},  vol. 51, pp. 2005--2035, 2013.



\bibitem{florent}
F.~Di~Meglio, R.~Vazquez, and M.~Krstic, ``Stabilization of a system of n+1 coupled first-order hyperbolic linear PDEs with a single boundary input,''  {\em IEEE Transactions on Automatic Control}, PP, 2013.

\bibitem{evans}
L.C.~Evans, {\em Partial Differential Equations}, AMS, Providence, Rhode Island, 1998.

\bibitem{dubljevic}
M.~Izadi, J.~Abdollahi, S.~S.~Dubljevic, ``PDE backstepping control of one-dimensional heat equation with time-varying domain,''
{\em Automatica}, vol. 54, pp. 41--48, 2015.

\bibitem{kugi}
L.~Jadachowski, T.~Meurer, A.~Kugi, ``Backstepping Observers for linear PDEs on Higher-Dimensional Spatial Domains.''
{\em Automatica}, vol. 51, pp. 85--97, 2015.

\bibitem{krstic}
M.~Krstic and A.~Smyshlyaev, {\em Boundary Control of PDEs},  SIAM, 2008

\bibitem{krstic5}
M.~Krstic, {\em Delay Compensation for nonlinear, Adaptive, and PDE Systems}, Birkhauser, 2009.

\bibitem{krstic3}
M.~Krstic and A.~Smyshlyaev, ``Backstepping boundary control for first order hyperbolic PDEs and application to systems with actuator and sensor delays,'' {\em Syst. Contr. Lett.}, vol. 57, pp. 750--758, 2008.

 \bibitem{chengkang-xie}
G.~Li and C.~Xie, ``Feedback stabilization of reaction-diffusion equation in a two-dimensional region,'' {\em Proceedings of the 2010 IEEE Conference on
Decision and Control (CDC)}, pp. 2985--2989, 2010.

\bibitem{meurer}
T.~Meurer and M.~Krstic, ``Finite-time multi-agent deployment: A nonlinear PDE motion planning approach,'' {\em Automatica}, vol. 47, pp. 2534--2542, 2011.

\bibitem{meurer2}
T.~Meurer, A.~Kugi, ``Tracking control for boundary controlled parabolic PDEs with varying parameters: combining backstepping and differential flatness,''
{\em Automatica}, vol. 45, n. 5, pp. 1182--1194, 2009.

\bibitem{scott}
S.J.~Moura, N.A.~Chaturvedi, and M.~Krstic, ``PDE estimation techniques for advanced battery management systems---Part I: SOC estimation,''  {\em Proceedings of the 2012 American Control Conference}, 2012.


\bibitem{jie}
J.~Qi, R.~Vazquez and M.~Krstic, ``Multi-Agent Deployment in 3-D via PDE Control,'' {\em  IEEE Transactions on Automatic Control}, in Press, 2015.

\bibitem{krstic2}
A.~Smyshlyaev, E.~Cerpa, and M.~Krstic, ``Boundary stabilization of a 1-D wave equation with in-domain antidamping,'' {\em SIAM J. Control Optim.}, vol. 48, pp. 4014--4031, 2010.

\bibitem{krstic4}
A.~Smyshlyaev and M.~Krstic, {\em Adaptive Control of Parabolic PDEs},  Princeton University Press, 2010.

\bibitem{triggiani}
R.~Triggiani, ``Boundary feedback stabilization of parabolic equations.{\em Appl. Math. Optimiz.},  vol. 6, pp. 201--220, 1980.

\bibitem{tsubakino}
D.~Tsubakino, S.~Hara, ``Backstepping observer design for parabolic PDEs with measurement of weighted spatial averages,''
{\em Automatica}, vol. 53, pp. 179--187, 2015.

\bibitem{vazquez}
R.~Vazquez and M.~Krstic, {\em Control of Turbulent and Magnetohydrodynamic Channel Flow}.  Birkhauser, 2008.

\bibitem{vazquez2}
R.~Vazquez and M.~Krstic, ``Control of 1-D parabolic PDEs with Volterra nonlinearities --- Part I: Design,'' {\em  Automatica}, vol. 44, pp. 2778--2790, 2008.


\bibitem{convloop}
R.~Vazquez and M.~Krstic, ``Boundary observer for output-feedback stabilization of thermal convection loop,'' {\em 
IEEE Trans. Control Syst. Technol.}, vol.18, pp. 789--797, 2010.

\bibitem{Vazquez2014}
R. Vazquez and M. Krstic.
\newblock Marcum {Q}-functions and explicit kernels for stabilization of linear hyperbolic systems with constant coefficients.
\newblock {\em Systems \& Control Letters }, 68:33--42, 2014

\bibitem{disk}
R.~Vazquez and M.~Krstic, ``Explicit boundary control of a reaction-diffusion equation on a disk,'' {\em  Proceedings of the 2014 IFAC World Congress}, 2014.

\bibitem{disk2}
R.~Vazquez and M.~Krstic, ``Explicit Boundary Control of Reaction-Diffusion PDEs on Arbitrary-Dimensional Balls,'' {\em  Proceedings of the 2015 European Control Conference}, 2015.

 \bibitem{vazquez-coron}
R.~Vazquez, E.~Trelat and J.-M.~Coron, ``Control for fast and stable laminar-to-high-Reynolds-numbers transfer in a 2D navier-Stokes channel flow,'' {\em Disc. Cont. Dyn. Syst. Ser. B}, vol. 10, pp. 925--956, 2008.


















\end{thebibliography}
\end{document}